\documentclass[12pt,reqno]{amsart}

\usepackage[T5,T1]{fontenc}
\usepackage{comment}

\includecomment{CollabNotes}

\specialcomment{IJMversion}{}{\ignorespaces}
\excludecomment{IJMversion}

\usepackage{amssymb}
\usepackage{amsfonts,amsmath,mathrsfs}
\usepackage{hyperref}
\usepackage{bm}
\usepackage{caption,subcaption}
\usepackage{color}
\usepackage{pgf,tikz}
\usepackage{wrap fig}
\usepackage{xypic}
\usepackage{epigraph}

\usepackage{marginnote}
\newcommand{\Br}{\text{\rm Br}}

\newcommand{\Gal}{\text{\rm Gal}}
\newcommand{\im}{\text{\rm im}}

\newcommand{\codim}{\text{\rm codim}}
\newcommand{\rk}{\text{\rm mult}}

\makeindex

\begin{document}


\thanks{The second author is partially supported by the Natural Sciences and Engineering Research Council of Canada grant R0370A01.  He also gratefully acknowledges the Faculty of Science Distinguished Research Professorship, Western University, in years 2004/2005 and 2020/2021. Finally he also gratefully acknowledges support of Western Academy for Advanced Research during the year 2022/2023. The fourth author is partially supported by 2021 Wellesley College Faculty Awards. The fifth author was supported  in part by National Security Agency grant MDA904-02-1-0061. The sixth author is funded by Vingroup Joint Stock Company and supported by Vingroup Innovation Foundation (VinIF) under the project code VINIF.2021.DA00030.}

\title[Quaternion algebras and square power classes]{Quaternion algebras and square power classes over biquadratic extensions}

\author[F.~Chemotti]{Frank Chemotti}
\email{fchemotti@gmail.com}

\author[J.~Min\'{a}\v{c}]{J\'{a}n Min\'{a}\v{c}}
\address{Department of Mathematics, Western University, London, Ontario, Canada N6A 5B7}
\email{minac@uwo.ca}

\author[T.T.~Nguyen]{Tung T.~Nguyen}
\address{Department of Mathematics, Western University, London, Ontario, Canada N6A 5B7}
\email{tungnt@uchicago.edu}

\author[A.~Schultz]{Andrew Schultz}
\address{Department of Mathematics, Wellesley College, 106 Central Street, Wellesley, MA \ 02481 \ USA}
\email{andrew.c.schultz@gmail.com}

\author[J.~Swallow]{John Swallow}
\address{Office of the President, Carthage College, 2001 Alford Park Drive, Kenosha, WI \ 53140 \ USA}
\email{jswallow@carthage.edu}

\author[N.D.~T\^an]{{\fontencoding{T5}\selectfont Nguy\~ \ecircumflex n Duy T\^an}}
\address{School of Applied Mathematics and Informatics, Hanoi University of Science and
Technology, Hanoi, Vietnam}
\email{tan.nguyenduy@hust.edu.vn}

\dedicatory{Dedicated to Professor Moshe Jarden with gratitude and admiration}

\date{\today}

\begin{abstract}
Recently the Galois module structure of square power classes of a field $K$ has been computed under the action of $\Gal(K/F)$ in the case where $\Gal(K/F)$ is the Klein $4$-group.  Despite the fact that the modular representation theory over this group ring includes an infinite number of non-isomorphic indecomposable types, the decomposition for square power classes includes at most $9$ distinct summand types.  In this paper we determine the multiplicity of each summand type in terms of a particular subspace of $\Br(F)$, and show that all ``unexceptional" summand types are possible.
\end{abstract}

\maketitle


\newtheorem*{theorem*}{Theorem}
\newtheorem*{lemma*}{Lemma}
\newtheorem{theorem}{Theorem}
\newtheorem{proposition}{Proposition}[section]
\newtheorem{corollary}[proposition]{Corollary}
\newtheorem{lemma}[proposition]{Lemma}

\theoremstyle{definition}
\newtheorem*{definition*}{Definition}
\newtheorem*{remark*}{Remark}
\newtheorem{example}[proposition]{Example}



\parskip=8pt plus 2pt minus 2pt

\section{Introduction}

\subsection{Motivation} 

For a field $K$ with $\text{\rm char}(K) \neq 2$, the square power classes $K^\times/K^{\times 2}$ are perhaps the most well-known parameterizing space for Galois extensions of $K$.  Indeed, it is an exercise in elementary Galois theory that the $n$-dimensional subspaces of the $\mathbb{F}_2$-vector space $K^\times/K^{\times 2}$ are in correspondence with the elementary $2$-abelian extensions of $K$ of rank $n$.  When $p$ is an odd prime number and $K$ contains a primitive $p$th root of unity, Kummer theory provides a generalization of this result, connecting $p$th power classes of $K$ to elementary $p$-abelian extensions of $K$.

More recently, these power classes (and several of their generalizations, including certain cohomological cousins) have received renewed interest in Galois modules.  One source of this interest comes from the observation that if $K/F$ is a Galois extension with $\Gal(K/F) = G$, then the $\mathbb{F}_p[G]$-modules of $K^\times/K^{\times p}$  are in correspondence with elementary $p$-abelian extensions of $K$ that are Galois over $F$.  In fact, if $L/K$ is the field extension corresponding to some $\mathbb{F}_p[G]$-submodule $N$, it is often possible to compute $\Gal(L/F)$ strictly in terms of the module-theoretic structure of $N$, together with a small amount of field-theoretic information attached to the elements in $N$.  The interested reader can consult \cite{Waterhouse} for some first results in this vein, and \cite{Schultz} for a generalization; in both of these cases, the group $G$ is assumed to be a finite, cyclic $p$-group.  In this way, it is typically possible to use power classes of a field as a parameterizing space for a much broader classes of Galois theoretic objects.  Said another way: the Galois module structure of power classes encodes a tremendous amount of information about Galois extensions of $K/F$, and for this reason the search for such decompositions is of great interest.

Quite a lot has been done in this area in the past twenty years, though results in the special case of local fields (see \cite{Bo,F}) have been known for more than 50 years.  The more recent work began with \cite{MS}
\begin{IJMversion}
 (which appeared in this journal in 2003)
\end{IJMversion}
 in which two of the authors provided a module decomposition of $p$th power classes in the case where $\Gal(K/F) \simeq \mathbb{Z}/p\mathbb{Z}$.  Since that time there has been a wealth of additional module decompositions.  Most of these results have stuck to the case where $\Gal(K/F)$ is a cyclic $p$ group, and they include: the structure of $K^\times/K^{\times p}$ when $\Gal(K/F) \simeq \mathbb{Z}/p^n\mathbb{Z}$ in \cite{MSS1}; the structure of associated cohomology groups in \cite{LMSS,LMS}; and characteristic $p$-analogs of both of these objects in \cite{BS,BLMS,Schultz}.  

The last year has seem some dramatic new developments in the study of Galois modules of this ilk, with module decompositions being determined in a number of cases where the modular representation theory allows for an infinite number of indecomposable types.  For example, the module structure for higher power classes $K^\times/K^{\times p^m}$ have been computed in the case where $\Gal(K/F) \simeq \mathbb{Z}/p^n\mathbb{Z}$ in \cite{MSS2b}.  However, there are also module computations in the case where $\Gal(K/F)$ is not a cyclic $p$-group.  This includes the decomposition of power classes (or their characteristic $p$-analogs) when $\Gal(K/F)$ is \emph{any} finite $p$-group, albeit under the assumption that the maximal pro-$p$ quotient of the absolute Galois group of $F$ is a finitely generated free pro-$p$ group and that either $F$ contains a primitive $p$th root of unity or $\text{char}(F) = p$ (see \cite{HMS}). It also includes the decomposition of square power classes in the case where $\Gal(K/F)$ is the Klein $4$-group in \cite{CMSS}.

Although \cite{CMSS} provides a list of the indecomposable types which can appear in the decomposition of $K^\times/K^{\times 2}$ when $\Gal(K/F) \simeq \mathbb{Z}/2\mathbb{Z}\oplus\mathbb{Z}/2\mathbb{Z}$ and characterizes the appearance of these summands in terms of certain simultaneous equations in the module, it does not give an explicit, arithmetic mechanism for determining the multiplicities of each summand type.  This paper works to fill this void, and in doing so makes novel contributions to our understanding of these power classes and their manifestations in the Galois-theoretic and arithmetic properties of the corresponding field.  In particular, we will use the Brauer group of $F$ --- including particular distinguished elements which are tied to the extension $K/F$ --- as our key tool for analyzing the decomposition of $K^\times/K^{\times 2}$.  The main results of this paper are Theorems \ref{th:main.unexceptional} and \ref{th:main.exceptional}, with Proposition \ref{prop:solving.equations.with.Brauer} playing a significant role along the way.  We showcase these results by determining module-theoretic invariants of certain biquadratic extensions of $\mathbb{Q}$ via Legendre symbols and the local-global principle.

\subsection{Statement of main results}

We first establish terminology and notation.  Throughout we will let $F$ denote a field with $\text{\rm char}(F) \neq 2$, and $K/F$ will be an extension with $\Gal(K/F) \simeq \mathbb{Z}/2\mathbb{Z}\oplus\mathbb{Z}/2\mathbb{Z}$.  We let $a_1,a_2 \in F$ be given so that $K = F(\sqrt{a_1},\sqrt{a_2})$, and we write $\sigma_1$ and $\sigma_2$ for their duals in $\text{Gal}(K/F)$; that is, we have $\sigma_i(\sqrt{a_j}) = (-1)^{\delta_{ij}}\sqrt{a_j}$.  For $i \in \{1,2\}$ we define $K_i = F(\sqrt{a_i})$.  Write $H_i$ for the subgroup of $\Gal(K/F)$ which fixes elements in $K_i$, and $\overline{G_i}$ for the corresponding quotient group: $\overline{G_i} := \Gal(K_i/F) = \{\overline{\text{\rm id}},\overline{\sigma_i}\}.$  In the same spirit, write $K_3 = F(\sqrt{a_1a_2})$, denote the subgroup of $\Gal(K/F)$ which fixes $K_3$ as $H_3$, and use $\overline{G_3}$ for the corresponding quotient $G/H_3 = \Gal(K_3/F)$.  For the sake of uniformity in notation, let $H_0 = \{\text{\rm id}\}$ (the elements which fix the extension $K/F$) and $H_4 = \Gal(K/F)$ (the elements which fix the extension $F/F$), and use $\overline{G_0}$ and $\overline{G_4}$ for their quotients. 

For convenience we will write $K^\times/K^{\times 2}$ as $J(K)$.  When $\gamma \in K^\times$, we use $[\gamma]$ to denote the class represented by $\gamma$ in $J(K)$.  Occasionally we will be interested in square power classes of some other field $L$, in which case we denote the set by $J(L)$, and denote its elements as $[\gamma]_L$.  

The decomposition of $J(K)$ provided by \cite{CMSS} is given in the following proposition.  Note that in addition to free modules over certain quotients of $\mathbb{F}_2[G]$, there are summands expressed in terms of $\mathbb{F}_2[G]$-modules $\Omega^1, \Omega^2, \Omega^{-1}$, and $\Omega^{-2}$.  These are indecomposable modules with $\dim\Omega^1 = \dim\Omega^{-1} = 3$ and $\dim\Omega^2 = \dim\Omega^{-2} = 5$; they are most easily described in terms of generators and relations.  For example, we have
\begin{align*}
\Omega^1 &= \left\langle a,b,c | a^{1+\sigma_1} = b = c^{1+\sigma_2}, a^{1+\sigma_2} = c^{1+\sigma_1} = 1\right\rangle\\
\Omega^2 &= \left\langle a,b,c,d,e | a^{1+\sigma_1} = b = c^{1+\sigma_2}, c^{1+\sigma_1} = d = e^{1+\sigma_2}, a^{1+\sigma_2} = e^{1+\sigma_1} = 1\right\rangle.
\end{align*}The interested reader can see \cite{CMSS} for a quick review of the basics of these modules, or could consult \cite[Theorem 4.3.3]{Benson} for complete details on the indecomposable types that appear over a generic $\mathbb{F}_2[\mathbb{Z}/2\mathbb{Z}\oplus\mathbb{Z}/2\mathbb{Z}]$-module.

\begin{proposition}\label{prop:CMSS.main}
\suppressfloats[t]
Suppose that $\text{\rm char}(K) \neq 2$ and that $\Gal(K/F) \simeq \mathbb{Z}/2\mathbb{Z}\oplus\mathbb{Z}/2\mathbb{Z}$.  Let $J(K) = K^\times/K^{\times 2}$.  Then as an $\mathbb{F}_2[\Gal(K/F)]$-module we have
$$J(K) \simeq X \oplus Y_0 \oplus Y_1 \oplus Y_2 \oplus Y_3 \oplus Y_4 \oplus Z_1\oplus Z_2,$$ where
\begin{itemize}
\item $X$ is isomorphic to one of $\{0\}, \mathbb{F}_2,\mathbb{F}_2\oplus\mathbb{F}_2, \Omega^{-1}, \Omega^{-2},$ or $\Omega^{-1}\oplus\Omega^{-1}$;
\item for each $i \in \{0,1,2,3,4\}$, the summand $Y_i$ is a direct sum of modules isomorphic to $\mathbb{F}_2[\overline{G_i}]$; and
\item for each $i \in \{1,2\}$, the summand $Z_i$ is a direct sum of modules isomorphic to $\Omega^i$.
\end{itemize}
\end{proposition}

We will refer to the $X$ summand from this theorem as the exceptional summand in the decomposition; the other summands will be called the unexceptional summands.  

There are two main results in this manuscript, both of which allow us to understand the module-theoretic invariants attached to $J(K)$ in terms of $\Br(F)$.  For $a,b \in F^\times$, we write $(a,b)$ for the class of the quaternion algebra in $\Br(F)$.  We will use additive notation for Brauer group elements.  Of particular importance in our results is the following subspace of $\Br(F)$: $$\mathcal{S} = \left\langle (a_1,a_1),(a_1,a_2),(a_2,a_2)\right\rangle_{\mathbb{F}_2}.$$  

It will be useful for us to note that if $c,d \in F^\times$ are equivalent in $J(K)$, then $(a_1,c)-(a_1,d) \in \mathcal{S}$, and likewise $(a_2,c)-(a_2,d) \in \mathcal{S}$.  To see that this is true, observe that if $c=dk^2$ for some $k \in K^\times$, then by Kummer theory we must have $k^2 = f^2 a_1^{\varepsilon_1}a_2^{\varepsilon_2}$, where $\varepsilon_1,\varepsilon_2 \in \{0,1\}$.  But then we have $$(a_1,c) - (a_1,d) = (a_1,df^2a_1^{\varepsilon_1}a_2^{\varepsilon_2}) -(a_1,d)= \varepsilon_1(a_1,a_1)+\varepsilon_2(a_1,a_2) \in \mathcal{S}.$$  A similar calculation shows that $(a_2,c)-(a_2,d) \in \mathcal{S}$.  For this reason, for a given $[f] \in [F^\times] \subseteq J(K)^G$, it is reasonable for us to examine whether $(a_1,f)$ and/or $(a_2,f)$ is in the subspace $\mathcal{S}$.

To describe our first main result, observe that each of the unexceptional summands  from Proposition \ref{prop:CMSS.main} is a direct sum of modules of a particular isomorphism type.  One is naturally interested in knowing the multiplicity of each summand.  The multiplicities of the unexceptional summands are determined in \cite{CMSS} by the codimension of certain subspaces of $[F^\times]$ that are phrased in terms of solvability of certain simultaneous equations.  When $V_1$ and $V_2$ are $\mathbb{F}_2$-subspaces with $V_1 \subseteq V_2$, we will write $\codim(V_1 \hookrightarrow V_2)$ to denote the codimension of $V_1$ within $V_2$.  In this paper, we will express the multiplicities of the unexceptional summands in terms of the following subspaces of $[F^\times]$:

\begin{equation}\label{eq:Brauer.subspaces}
\begin{gathered}
A = [N_{K/F}(K^\times)] \qquad 
B = \{[f]: (a_1,f) \in \mathcal{S}\}\\ 
C = \{[f]: (a_2,f) \in \mathcal{S}\}\qquad 
D = \{[f] : (a_1a_2,f) \in \mathcal{S}\}\\
V = B \cap C \qquad  
W_B = \{[f] \in B: \exists [g] \in C \text{ so that }(a_2,f)+(a_1,g) \in \mathcal{S}\}\\
W_C = \{[g] \in C: \exists[f] \in B \text{ so that }(a_2,f)+(a_1,g) \in \mathcal{S}\}\\
W_D = \{[f][g]: [f] \in B, [g] \in C \text{ and }(a_2,f)+(a_1,g) \in \mathcal{S}\}.
\end{gathered}
\end{equation}

It is clear from these definitions that $V$ is a subspace of each of $B$, $C$, $D$, $W_B,W_C$, and $W_D$; that $W_B$ is a subspace of $B$; that $W_C$ is a subspace of $C$; and that $W_D$ is a subspace of $D$.  One can see easily that $B \cap D = C \cap D = V$.  Since it is well known that $(a_i,f) = 0$  when $f \in N_{K_i/F}(K_i^\times)$, we also have $A$ is a subspace of $V$.

\begin{theorem}\label{th:main.unexceptional}
Let $\varepsilon=1$ in the case where $[N_{K/K_1}(K^\times)] \cap [N_{K/K_2}(K^\times)] \subseteq [F^\times]$, and otherwise let $\varepsilon=0$.
Then the multiplicities of the unexceptional summands $Y_0,Y_1,Y_2,Y_3,Y_4,Z_1$, and $Z_2$ from Proposition \ref{prop:CMSS.main} are: 
$\rk(Y_0) = \dim A$; $\rk(Z_1) = \codim(A \hookrightarrow V)$; $\rk(Z_2) = \codim(V \hookrightarrow W_B) = \codim(V \hookrightarrow W_C)$;
$\rk(Y_1)=\codim (W_B \hookrightarrow B)$; $\rk(Y_2)=\codim (W_C \hookrightarrow C)$;
$\rk(Y_3)=\codim (W_D \hookrightarrow D)$;
$\rk(Y_4)=\codim( B+C+D \hookrightarrow [F^\times])-\varepsilon.$
\end{theorem}

Our second main result allows us to determine the module structure of the exceptional summand in terms of the subspace $\mathcal{S}$, together with one additional piece of data.

\begin{theorem}\label{th:main.exceptional}
Let $\varepsilon=1$ in the case where $[N_{K/K_1}(K^\times)] \cap [N_{K/K_2}(K^\times)] \subseteq [F^\times]$, and otherwise let $\varepsilon=0$.
Then the exceptional summand $X$ from Proposition \ref{prop:CMSS.main} is isomorphic to:
\begin{itemize}
\item $\{0\}$ in the case where $\dim(\mathcal{S}) = 3$;
\item $\mathbb{F}_2$ in the case where $\dim(\mathcal{S}) = 2$;
\item $\Omega^1$ or $\mathbb{F}_2 \oplus \mathbb{F}_2$ in the case where $\dim(\mathcal{S})=1$, with the former occuring precisely when $(a_1,a_1),(a_1,a_2),$ and $(a_2,a_2)$ satisfy one of the following: $(a_1,a_2)=(a_2,a_2)=0$; or $(a_1,a_1)=(a_1,a_2)=0$; or $(a_1,a_1)+(a_1,a_2)=(a_1,a_2)+(a_2,a_2)=0$;
\item $\Omega^{-1} \oplus \Omega^{-1}$ or $\Omega^{-2}$ in the case where $\dim(\mathcal{S})=0$, with the former occuring when $\varepsilon=1$ and the latter when $\varepsilon = 0$.
\end{itemize}
\end{theorem}

These two results are siblings of each other, born from the marriage that exists between the vanishing of certain Brauer group elements and the solvability of associated Galois embedding problems.  Indeed, we will see that the following result plays the critical role in the proofs of our theorems.

\begin{proposition}[cf. {\cite[(7.6)]{Frolich}}, {\cite[Cor.~2.6]{Ledet}},{\cite[Th.~1.2]{MinacSmith}}]\label{prop:Frolich}
Suppose that $E$ is a field with $\text{\rm char}(E) ]\neq 2$, and that $L/E$ is a Galois extension with $G=\Gal(L/E) = \langle \sigma_1,\cdots,\sigma_n\rangle \simeq \left(\mathbb{Z}/2\mathbb{Z}\right)^{\oplus n}$.  Suppose further that $L=E(\sqrt{a_1},\cdots,\sqrt{a_n})$, with $\sigma_j(\sqrt{a_i})=(-1)^{\delta_{ij}}\sqrt{a_i}$.  Let $c_{ij} \in \{0,1\}$ with $1 \leq i \leq j \leq n$ be given such that $$\widetilde{G} = \left\langle \omega,\widetilde{\sigma_1},\cdots,\widetilde{\sigma_n} \Big{|} [\omega,\widetilde{\sigma_i}]=\text{\rm id}, \widetilde{\sigma_i}^2=\omega^{c_{ii}}, [\widetilde{\sigma_i},\widetilde{\sigma_j}]=\omega^{c_{ij}}\right\rangle$$ is a non-split, central group extension of $G$ by $\langle \omega \rangle \simeq \mathbb{Z}/2\mathbb{Z}$.  Then $\prod_{i \leq j} (a_i,a_j)^{c_{ij}}$ is trivial in $\Br(E)$ if and only if the Galois-theoretic embedding problem corresponding to $\widetilde{G} \twoheadrightarrow G$ is solvable over $L/E$.
\end{proposition}

\subsection{Outline of paper}

In section \ref{sec:ranks.of.nonexceptional.summands} we will review how the summands $Y_0,Y_1,Y_2,Y_3,Y_4,Z_1,$ and $Z_2$ from Proposition \ref{prop:CMSS.main} are determined by the solvability of certain simultaneous equations in \cite{CMSS}.  Then we will connect the solvability of these equations to the subspaces from Equation \ref{eq:Brauer.subspaces}. This will give the proof of the Theorem \ref{th:main.unexceptional}. In section \ref{sec:exceptional.summand} we will review how the construction of $X$ from \cite{CMSS} hinges on the image of a particular function $T:J(K)^G \to \bigoplus_{i=1}^3 (K_i^\times \cap K^{\times 2})/K_i^{\times 2}$.  We will then see how the image of that function can be encoded in the vanishing of elements from $\mathcal{S}$; this will give the proof Theorem \ref{th:main.exceptional}.  In the final section, we will consider several explicit examples.  Along the way, we will be able to show that each of the unexceptional summand types from Proposition \ref{prop:CMSS.main} are realizable.

\subsection{Acknowledgements}

The authors gratefully acknowledge an anonymous referee's insightful and encouraging report which included helpful suggestions for the improvement of the exposition of this paper.

\section{Multiplicities of unexceptional summands}\label{sec:ranks.of.nonexceptional.summands}

In \cite{CMSS}, the decomposition provided by Proposition \ref{prop:CMSS.main} begins by constructing certain subspaces of $[F^\times]$ and uses them to produce modules that ultimately become $Y_0,Y_1,Y_2,Y_3,Y_4, Z_1$, and $Z_2$.  We will exhibit these subspaces by using diagrams to depict various equations; in these drawings, if we draw an arrow between two elements, then this indicates that the operator acts on the source of the arrow to produce the terminus of the arrow. (To help cut down on the size of diagrams, we will adopt the convention that an element which has no arrows emanating from it is to be considered fixed under the action of both $\sigma_1$ and $\sigma_2$.)   A diagram that includes multiple arrows is then a system of simultaneous equations.  For example, 
$$
\begin{tikzpicture}[scale=.65]
\node (A3) at (0,2) {$[\alpha]$};
\node (A1) at (-2,0) {$[\beta]$};
\node (A2) at (2,0) {$[\gamma]$};

\draw[->] (A3) to node[sloped,above] {\tiny $1+\sigma_2$}(A1);
\draw[->] (A3) to node[sloped,above] {\tiny $1+\sigma_1$}(A2);
\end{tikzpicture}
$$ encodes the simultaneous equations $[\alpha]^{1+\sigma_2} = [\beta]$ and $[\alpha]^{1+\sigma_1} = [\gamma]$, where $[\beta]$ and $[\gamma]$ are both fixed elements.   The relevant subspaces are denoted $\mathscr{A},\mathscr{B},\mathscr{C},\mathscr{D},\mathscr{V},\mathscr{W}_B,$ and $\mathscr{W}_C$, and are shown in Figure \ref{fig:first.filtration.diagrams}.  

\begin{figure}[!ht]
\begin{tikzpicture}[scale=.65]
\node (A3) at (0,2) {$[k]$};
\node (A1) at (-2,0) {$[k]^{1+\sigma_2}$};
\node (A2) at (2,0) {$[k]^{1+\sigma_1}$};
\node (A0) at (0,-2) {$[f]$};
\node at (0,-4) {Diagram for $\mathscr{A}$};

\draw[->] (A3) to node[sloped,above] {\tiny $1+\sigma_2$}(A1);
\draw[->] (A3) to node[sloped,above] {\tiny $1+\sigma_1$}(A2);
\draw[->] (A1) to node[sloped,below] {\tiny $1+\sigma_1$}(A0);
\draw[->] (A2) to node[sloped,below] {\tiny $1+\sigma_2$}(A0);

\draw[dashed] (5,2) -- (5,-5);

\node (B1) at (10,0) {$[\gamma_1]$};
\node (B2) at (14,0) {$[\gamma_2]$};
\node (B0) at (12,-2) {$[f]$};
\node (Bt1) at (8,-2) {$[1]$};
\node (Bt2) at (16,-2) {$[1]$};
\node at (12,-4) {Diagram for $\mathscr{V}$};

\draw[->] (B1) to node[sloped,above] {\tiny $1+\sigma_1$}(B0);
\draw[->] (B2) to node[sloped,above] {\tiny $1+\sigma_2$}(B0);
\draw[->] (B1) to node[sloped,above] {\tiny $1+\sigma_2$}(Bt1);
\draw[->] (B2) to node[sloped,above] {\tiny $1+\sigma_1$}(Bt2);

\draw[dashed] (-4,-5) -- (20,-5);

\node (C1) at (0,-6) {$[\gamma]$};
\node (Ct) at (-2,-8) {$[1]$};
\node (C2) at (2,-8) {$[f]$};
\node at (0,-10) {Diagram for $\mathscr{B}$};

\draw[->] (C1) to node[sloped,above] {\tiny $1+\sigma_1$}(C2);
\draw[->] (C1) to node[sloped,above] {\tiny $1+\sigma_2$}(Ct);

\draw[dashed] (4,-5) -- (4,-11);

\node (D1) at (8,-6) {$[\gamma]$};
\node (D2) at (6,-8) {$[f]$};
\node (Dt) at (10,-8) {$[1]$};
\node at (8,-10) {Diagram for $\mathscr{C}$};

\draw[->] (D1) to node[sloped,above] {\tiny $1+\sigma_2$}(D2);
\draw[->] (D1) to node[sloped,above] {\tiny $1+\sigma_1$}(Dt);

\draw[dashed] (12,-5) -- (12,-11);

\node (E1) at (16,-6) {$[\gamma]$};
\node (E2) at (14,-8) {$[f]$};
\node (E3) at (18,-8) {$[f]$};
\node at (16,-10) {Diagram for $\mathscr{D}$};

\draw (E1) to node[sloped,above] {\tiny $1+\sigma_2$}(E2);
\draw (E1) to node[sloped,above] {\tiny $1+\sigma_1$}(E3);

\draw[dashed] (-4,-11) -- (20,-11);

\node (W1) at (4,-12) {$[\gamma_1]$};
\node (W2) at (8,-12) {$[\gamma_2]$};
\node (W3) at (12,-12) {$[\gamma_3]$};
\node (Wb0) at (2,-14) {$[1]$};
\node (Wb1) at (6,-14) {$[f]$};
\node (Wb2) at (10,-14) {$[g]$};
\node (Wb3) at (14,-14) {$[1]$};
\node at (8,-16) {Diagram for $\mathscr{W}_B$ and $\mathscr{W}_C$};

\draw[->] (W1) to node[sloped,above] {\tiny $1+\sigma_2$}(Wb0);
\draw[->] (W2) to node[sloped,above] {\tiny $1+\sigma_2$}(Wb1);
\draw[->] (W3) to node[sloped,above] {\tiny $1+\sigma_2$}(Wb2);
\draw[->] (W1) to node[sloped,above] {\tiny $1+\sigma_1$}(Wb1);
\draw[->] (W2) to node[sloped,above] {\tiny $1+\sigma_1$}(Wb2);
\draw[->] (W3) to node[sloped,above] {\tiny $1+\sigma_1$}(Wb3);

\end{tikzpicture}
\caption{Diagrams that represent the various systems of equations that are solvable in order for $[f]$ to be an element of the subspaces $\mathscr{A},\mathscr{B},\mathscr{C},\mathscr{D},$ and $\mathscr{V}$.  For $[f]$ to be in $\mathscr{W}_B$, there must exist some $[g]$ which satisfies the last diagram; symmetrically, for $[g]$ to be in $\mathscr{W}_C$, there must be some $[f]$ which satisfies the last diagram.}
\label{fig:first.filtration.diagrams}
\end{figure}
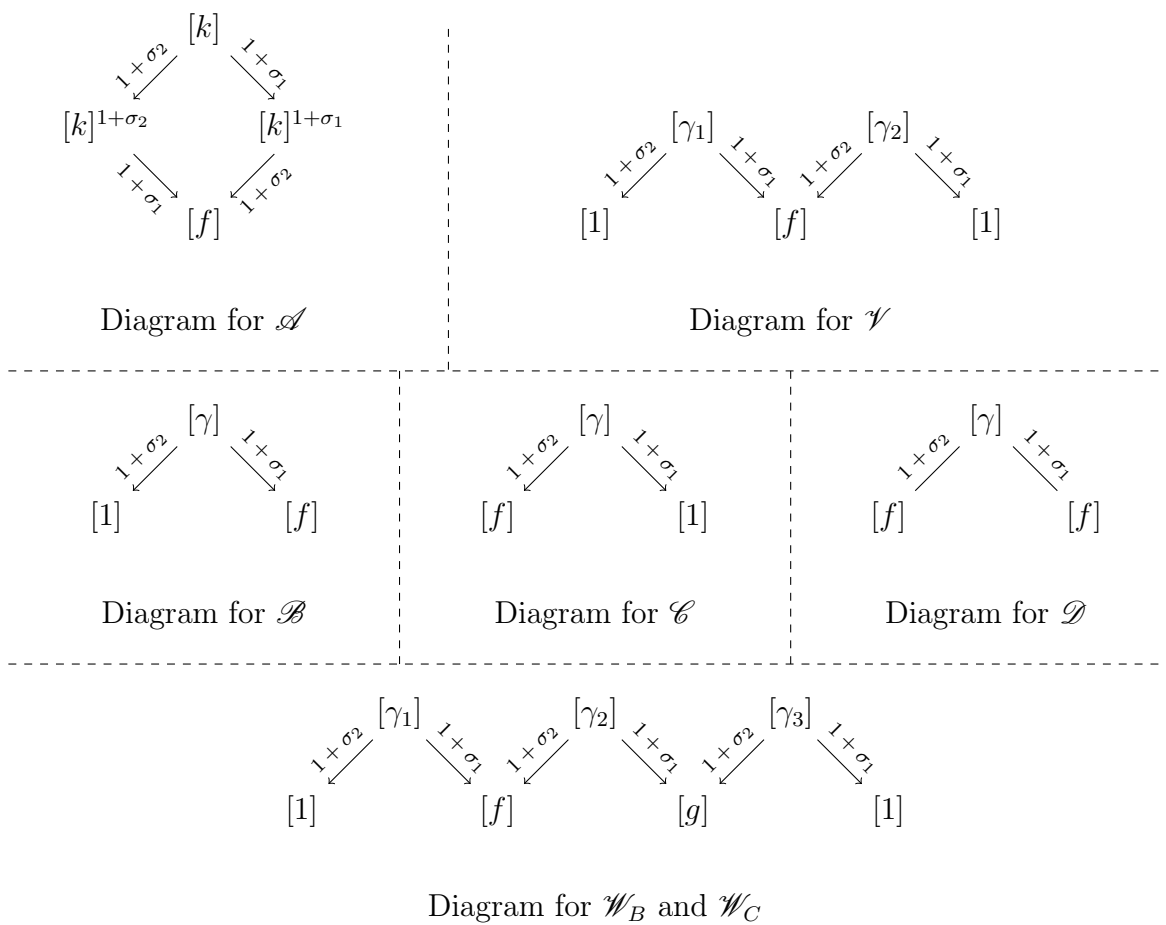		

We now recall the essential details of the construction of the unexceptional modules from \cite{CMSS}.  Readers who are interested in more details can consult  \cite[Th.~2]{CMSS} and the first two paragraphs of the proof of \cite[Th.~1]{CMSS}.

The module $Y_0$ is constructed by selecting a basis for $\mathscr{A}$, and for each basis element $[x]$ constructing a free module $M_x = \langle [k_x] \rangle$, where here $[k_x] \in [K^\times]$ satisfies $[N_{K/F}(k_x)]=[x]$.  One then defines $Y_0 = \bigoplus_{[x]} M_x$.  From this it is clear that $\rk(Y_0) = \dim(\mathscr{A})$.

The module $Z_1$ is constructed by selecting a basis for a complement to $\mathscr{A}$ within $\mathscr{V}$, and for each basis element $[x]$ constructing a module $M_x = \langle [\gamma_{1,x}],[\gamma_{2,x}]\rangle$ which satisfies the diagram for $\mathscr{V}$ (with $[\gamma_{i,x}]$ playing the role of $[\gamma_i]$ in the diagram).  Each $M_x$ is then isomorphic to $\Omega^{1}$, and then $Z_1 = \bigoplus_{[x]} M_x$.  From this one sees that $\rk(Z_1) = \codim(\mathscr{A} \hookrightarrow \mathscr{V})$.

The module $Z_2$ is constructed by first selecting complements to $\mathscr{V}$ within $\mathscr{W}_B$ and $\mathscr{W}_C$; call these complements $B_0$ and $C_0$.  One proves (see \cite[Lem.~3.2]{CMSS}) that there exists an isomorphism $\phi:B_0 \to C_0$ so that for each $[b] \in B_0$, the element $\phi([b]) \in C_0$ is the unique element $[c] \in C_0$ so that the diagram for $\mathscr{W}_B$ and $\mathscr{W}_C$ is solvable with $[f] = [b]$ and $[g] = [c]$.  (In particular, this gives $\codim(\mathscr{V} \hookrightarrow \mathscr{W}_B) = \codim(\mathscr{V} \hookrightarrow \mathscr{W}_C)$.)  To construct $Z_2$, one selects a basis for $B_0$, and then for each $[x]$ in the basis one constructs a module $M_x = \langle [\gamma_{1,x}],[\gamma_{2,x}],[\gamma_{3,x}]\rangle$, where the elements $[\gamma_{1,x}],[\gamma_{2,x}],[\gamma_{3,x}]$ are selected so that they solve the diagram for $\mathscr{W}_B$ and $\mathscr{W}_C$ with $[f] = [x]$ and $[g] = \phi([x])$ (and, of course, $[\gamma_{i,x}] = [\gamma_i]$).  Then $M_x \simeq \Omega^2$, and one defines $Z_2 = \bigoplus_{[x]} M_x$.  Hence we have $\rk(Z_2) = \codim(\mathscr{V} \hookrightarrow \mathscr{W}_B) = \codim(\mathscr{V} \hookrightarrow \mathscr{W}_C)$.

The constructions of $Y_1$, $Y_2$, and $Y_3$ are all produced in a way analogous to each other.  For instance, to construct $Y_1$ one selects a basis for a complement of $\mathscr{W}_B$ within $\mathscr{B}$.  For each basis element $[x]$, one then produces a module $M_x = \langle [\gamma_x] \rangle$ which satisfies the diagram for $\mathscr{B}$ (again, with $[\gamma_x]$ in place of $[\gamma]$ and $[x]$ in place of $[f]$).  Each such module is isomorphic to $\mathbb{F}_p[\overline{G_1}]$, and then one defines $Y_1 = \bigoplus_{[x]} M_x$.  Hence $\rk(Y_1) = \codim(\mathscr{W}_B \hookrightarrow \mathscr{B})$. (The construction for $Y_2$ follows this same script after replacing $\mathscr{W}_B$ with $\mathscr{W}_C$, $\mathscr{B}$ with $\mathscr{C}$, and $\overline{G_1}$ with $\overline{G_2}$. For the construction of $Y_3$, one replaces $\mathscr{W}_B$ with $(\mathscr{B} +\mathscr{C}) \cap \mathscr{D}$, replaces $\mathscr{B}$ with $\mathscr{D}$, and replaces $\overline{G_1}$ with $\overline{G_3}$.  For this last result, \cite[Lem.~3.1]{CMSS} shows that $[b][c] \in (\mathscr{B}+\mathscr{C}) \cap \mathscr{D}$ if and only if the last diagram in Figure \ref{fig:first.filtration.diagrams} is solvable for $[f]=[b]$ and $[g]=[c]$.)


Finally, to construct $Y_4$, one first must determine whether $[N_{K/K_1}(K^\times)] \cap [N_{K/K_2}(K^\times)] \subseteq [F^\times]$ or not.  If it is, then $Y_4$ is simply a complement to $\mathscr{C}+\mathscr{D}+\mathscr{E}$ within $[F^\times]$, and so $\rk(Y_4) = \codim(\mathscr{C}+\mathscr{D}+\mathscr{E} \hookrightarrow [F^\times])$.  Otherwise $Y_4$ is constructed by identifying a particular $1$-dimensional subspace $\mathscr{X}$ of $[F^\times]$ related to the exceptional summand $X$ that sits outside of $\mathscr{C}+\mathscr{D}+\mathscr{E}$, and then letting $Y_4$ be a complement to $\mathscr{C}+\mathscr{D}+\mathscr{E}+\mathscr{X}$ within $[F^\times]$.  Hence in this latter case, we have $\rk(Y_4) = \codim(\mathscr{C}+\mathscr{D}+\mathscr{E}\hookrightarrow [F^\times])-1$.

With these details in mind, the proof of Theorem 1 boils down to equating the subspaces $\mathscr{B},\mathscr{B},\mathscr{C},\mathscr{D},\mathscr{V},\mathscr{W}_B,\mathscr{W}_C$, and $(\mathscr{B}+\mathscr{C})\cap \mathscr{D}$ with the subspaces $A,B,C,D,V,W_B,W_C,$ and $W_D$ from \ref{eq:Brauer.subspaces}.  Each of these desired equalities follows from the following

\begin{proposition}\label{prop:solving.equations.with.Brauer}
Suppose that $f,g \in F^\times$ are given.  Then there exists $\gamma \in K^\times$ so that 
\begin{equation}\label{eq:solvable.diagram}
\begin{tikzpicture}[scale=.65]
\node (A3) at (0,2) {$[\gamma]$};
\node (A1) at (-2,0) {$[g]$};
\node (A2) at (2,0) {$[f]$};

\draw[->] (A3) to node[sloped,above] {\tiny $1+\sigma_2$}(A1);
\draw[->] (A3) to node[sloped,above] {\tiny $1+\sigma_1$}(A2);
\end{tikzpicture}
\end{equation} if and only if $(a_1,f)+(a_2,g) \in \mathcal{S}$.
\end{proposition}
  
\begin{proof}
The proof requires case analysis of $\dim_{\mathbb{F}_2} \langle [f],[g] \rangle$, but we will focus on the case where $[f]$ and $[g]$ are independent; other cases involve analogous arguments to the one detailed here.

Suppose first that $(a_1,f)+(a_2,g) \in \mathcal{S}$.  Hence there exist $c_1,c_2,c_3 \in \mathbb{F}_2$ so that $$(a_1,f)+(a_2,g)+c_1(a_1,a_1)+c_2(a_1,a_2)+c_3(a_2,a_2)=0.$$  Let $L = K(\sqrt{f},\sqrt{g}) = F(\sqrt{a_1},\sqrt{a_2},\sqrt{f},\sqrt{g})$.  By our assumption we have $\Gal(L/F) \simeq \left(\mathbb{Z}/2\mathbb{Z}\right)^{\oplus 4}$.  We abuse notation and continue to write $\sigma_1$ and $\sigma_2$ for the elements of $\Gal(L/F)$ that are dual to $a_1$ and $a_2$, and we also let $\sigma_f$ and $\sigma_g$ be dual to $f$ and $g$.  By Proposition \ref{prop:Frolich}, the Galois embedding problem that corresponds to the central, non-split group extension  
$$\widetilde{G} = \left\langle \omega,\widetilde{\sigma_1},\widetilde{\sigma_2},\widetilde{\sigma_f},\widetilde{\sigma_g} \Big{|} \widetilde{\sigma_1}^2 = \omega^{c_1}, \widetilde{\sigma_2}^2 = \omega^{c_3}, [\widetilde{\sigma_f},\widetilde{\sigma_1}] = [\widetilde{\sigma_g},\widetilde{\sigma_2}]=\omega, [\widetilde{\sigma_1},\widetilde{\sigma_2}] = \omega^{c_2}\right\rangle \twoheadrightarrow \left(\mathbb{Z}/2\mathbb{Z}\right)^{\oplus 4}$$
is solvable over $L/F$ (here all unstated relations are trivial).  We will let $\tilde L/F$ be a solution to this embedding problem.

Consider the subfield of $\tilde L/F$ which is fixed by $\langle \widetilde{\sigma_f},\widetilde{\sigma_g}\rangle$.  This is a degree $8$ extension which contains $K$, and hence can be viewed as a degree $2$ extension of $K$.  By Kummer theory we can write this field as $K(\sqrt{\gamma})$ for some $\gamma \in K$.  We will show that this $\gamma$ solves the simultaneous equations represented by Equation \ref{eq:solvable.diagram}.  The proof that $[\gamma]^{1+\sigma_2} = [g]$ is completely analogous to the proof that $[\gamma]^{1+\sigma_1} = [f]$, and so we will focus on the latter.  Our general strategy will be to first argue that $\gamma^{1+\sigma_1} \in L^{\times 2} \cap K_2^{\times}$, from which we deduce that $[\gamma]_{K_2}$ sits in a particular $3$-dimensional subspace of $J(K_2)$.  Some further calculations will allow us to identify $\gamma^{1+\sigma_1}$ as one of two possible elements of $J(K_2)$, at which point we will be able to deduce $[\gamma]^{1+\sigma_1} = [f]$.

Note first that $\omega$ fixes all of $K$, and hence permutes the roots of $X^2-\gamma=0$.  On the other hand, since $\sqrt{\gamma} \not\in L$ we have $\sqrt{\gamma}^{\omega} = - \sqrt{\gamma}$.  Since $\omega$ is central in $\widetilde{G}$, we therefore have $$\left(\sqrt{\gamma}\right)^{(1+\sigma_1)\omega} = \left(\sqrt{\gamma}^\omega\right)^{1+\sigma_1} = \left(- \sqrt{\gamma}\right)^{1+\sigma_1} = \sqrt{\gamma}^{1+\sigma_1}.$$ From this we have $\sqrt{\gamma}^{1+\sigma_1}$ is fixed by $\omega$, and hence $\sqrt{\gamma}^{1+\sigma_1} \in L^\times$.  Therefore $\gamma^{1+\sigma_1} \in L^{\times 2}$.  On the other hand, since $\widetilde{\sigma_1} = \sigma_1$ for elements of $L$, we have $$\left(\gamma^{1+\sigma_1}\right)^{\sigma_1} = \gamma^{\sigma_1+\sigma_1^2} = \gamma^{\sigma_1+1}.$$ Hence $\gamma^{1+\sigma_1}$ is in the subfield of $K$ fixed by $\sigma_1$, so that $\gamma^{1+\sigma_1} \in K_2$.  An application of Kummer theory then tells us that $[\gamma^{1+\sigma_1}]_{K_2} \in \frac{L^{\times 2} \cap K_2^\times}{K_2^{\times 2}} = \langle [a_1]_{K_2},[f]_{K_2},[g]_{K_2}\rangle_{\mathbb{F}_2}$.

To determine the value of $[\gamma^{1+\sigma_1}]_{K_2}$ more precisely, we will use the fact that $\widetilde{\sigma_f}\omega\widetilde{\sigma_1} = \widetilde{\sigma_1}\widetilde{\sigma_f}$, as well as the fact that $\widetilde{\sigma_f}$ permutes the roots of $X^2-\gamma$, and that $\sqrt{\gamma}^\omega = -\sqrt{\gamma}$.  With these facts in mind, we have 
$$\left(\sqrt{\gamma}\right)^{(1+\widetilde{\sigma_1})\widetilde{\sigma_f}} = \left(\sqrt{\gamma}^{\widetilde{\sigma_f}}\right)^{1+\omega\widetilde{\sigma_1}} =  \left(\pm \sqrt{\gamma}\right)^{1+\omega\widetilde{\sigma_1}} = 
\left(\pm \sqrt{\gamma}\right)\left(\mp\sqrt{\gamma}\right)^{\widetilde{\sigma_1}} = -\sqrt{\gamma}^{1+\widetilde{\sigma_1}}.$$  Since $\sqrt{\gamma}^{1+\widetilde{\sigma_1}}$ is not fixed by $\widetilde{\sigma_f}$, this forces $[\gamma^{1+\sigma_1}]_{K_2} \not\in \langle[a_1]_{K_2},[g]_{K_2}\rangle$. Hence $[\gamma^{1+\sigma_1}]_{K_2} \in [f]_{K_2}\langle [a_1]_{K_2},[g]_{K_2}\rangle$.

On the other hand, a similar argument (this time using the fact that $\widetilde{\sigma_g}\widetilde{\sigma_1} = \widetilde{\sigma_1}\widetilde{\sigma_g}$ and that $\sqrt{\gamma}^{\widetilde{\sigma_g}} = \pm \sqrt{\gamma}$) gives $$\left(\sqrt{\gamma}\right)^{(1+\widetilde{\sigma_1})\widetilde{\sigma_g}} = \left(\sqrt{\gamma}^{\widetilde{\sigma_g}}\right)^{1+\widetilde{\sigma_1}} = \left(\pm \sqrt{\gamma}\right)^{1+\widetilde{\sigma_1}} = \sqrt{\gamma}^{1+\widetilde{\sigma_1}}.$$  Hence $\sqrt{\gamma}^{1+\widetilde{\sigma_1}}$ must be fixed by $\widetilde{\sigma_g}$, and so $[\gamma^{1+\sigma_1}]_{K_2} \in \{[f]_{K_2},[a_1f]_{K_2}\}$.  In either case, we may use the fact that $[a_1] = [1]$ in $J(K)$ to conclude that $[\gamma]^{1+\sigma_1} = [f]$.

Now we prove the opposite implication.  Start by assuming we have some $[\gamma] \in J(K)$ which solves Equation \ref{eq:solvable.diagram}.  The extension $\widetilde{L} = K(\sqrt{\gamma},\sqrt{f},\sqrt{g})$ is Galois over $F$ (recall from the introduction that $\mathbb{F}_2[\Gal(K/F)]$-submodules of $J(K)$ correspond to elementary $2$-abelian extensions of $K$ that are Galois over $F$).  Now since $[f],[g] \in [F^\times]$ we have $L = K(\sqrt{f},\sqrt{g})$ has $\Gal(L/F) \simeq \left(\mathbb{Z}/2\mathbb{Z}\right)^{\oplus 4}$; since $[\gamma] \not\in [F^\times]$ (lest $[\gamma]^{1+\sigma_i} = [1]$), it must be the case that $\Gal(\widetilde{L}/F)$ is a central, non-split extension of $\left(\mathbb{Z}/2\mathbb{Z}\right)^{\oplus 4}$.  We will compute enough information about $\Gal(\widetilde{L}/F)$ (in terms of generators and relations) to understand this group extension through the lens of Proposition \ref{prop:Frolich}, at which point we will be able to argue $(a_1,f)+(a_2,g) \in \mathcal{S}$.

As before, let $\sigma_1,\sigma_2,\sigma_f,$ and $\sigma_g$ be elements of $\Gal(L/F)$ dual to $\sqrt{a_1},\sqrt{a_2},\sqrt{f}$, and $\sqrt{g}$ (respectively); hence $\Gal(L/F)=\langle \sigma_1,\sigma_2,\sigma_f,\sigma_g\rangle$.  We will write $\widetilde{\sigma_1},\widetilde{\sigma_2},\widetilde{\sigma_f},\widetilde{\sigma_g}$ for lifts of these elements to $\Gal(\widetilde{L}/F)$.  The other generator of $\Gal(\widetilde{L}/F)$ will be the central element $\omega$ which generates $\Gal(\widetilde{L}/L)$; this just means that $\omega$ fixes elements of $L$, but $\sqrt{\gamma}^\omega = -\sqrt{\gamma}$.  We have $\Gal(\widetilde{L}/F) = \langle \widetilde{\sigma_1}, \widetilde{\sigma_2}, \widetilde{\sigma_f}, \widetilde{\sigma_g}, \omega\rangle$. Since $\omega$ generates the kernel of the surjection $\Gal(\widetilde{L}/F) \twoheadrightarrow \Gal(L/F)$, there must be $c_1,c_2,c_3 \in \{0,1\}$ so that $$\widetilde{\sigma_1}^2 = \omega^{c_1}, [\widetilde{\sigma_1},\widetilde{\sigma_2}]=\omega^{c_2}, \widetilde{\sigma_2}^2 = \omega^{c_3}.$$

We claim that $\widetilde{\sigma_f}^2 = \text{\rm id}$.  To see this, note that the element certainly fixes all of $L$ (since it reduces to $\sigma_f^2$ there). Since $\gamma \in K^\times$ it must be fixed by $\sigma_f$; hence we have  $$\left(\sqrt{\gamma}\right)^{1-\widetilde{\sigma_f}^2} = \left(\sqrt{\gamma}^{1+\widetilde{\sigma_f}}\right)^{1-\widetilde{\sigma_f}} = \left(\pm \sqrt{\gamma^{1+\sigma_f}}\right)^{1-\widetilde{\sigma_f}} = \gamma^{1-\sigma_f} = 1.$$   Therefore $\widetilde{\sigma_f}^2$ fixes the generators of $\widetilde{L}$, and so must be trivial.  This same line of reasoning gives $\widetilde{\sigma_g}^2 = \text{\rm id}$, as well as $(\widetilde{\sigma_f}\widetilde{\sigma_g})^2 = \text{\rm id}$.  One then deduces that $[\widetilde{\sigma_f},\widetilde{\sigma_g}] = \text{\rm id}$.

We now show that $[\widetilde{\sigma_f},\widetilde{\sigma_1}] = \omega$.  First, note that $[\widetilde{\sigma_f},\widetilde{\sigma_1}]$ fixes each of $\sqrt{f},\sqrt{g},\sqrt{a_1},\sqrt{a_2}$ since each of these elements is in $L$, so that this element acts as $[\sigma_f,\sigma_1]=\text{\rm id}$.  Hence all that is left to argue is that $\sqrt{\gamma}^{[\widetilde{\sigma_f},\widetilde{\sigma_1}]} = -\sqrt{\gamma}$.  Note that $\widetilde{\sigma_f}$ permutes the roots of $X^2-\gamma$ since $\gamma \in K$; let $\varepsilon \in \{0,1\}$ be defined so that $\sqrt{\gamma}^{\widetilde{\sigma_f}} = (-1)^\varepsilon \sqrt{\gamma}$.  On the other hand, since we are assuming $[\gamma]^{1+\sigma_1}=[f]$, we know that there exists some $k \in K^\times$ so that $\sigma_1$ sends the roots of $X^2-\gamma$ to roots of $X^2-f\gamma k^2$.  We will select roots of these polynomials so that $\sqrt{\gamma}^{\widetilde{\sigma_1}} = \sqrt{f}\sqrt{\gamma}k$.  We then have 
\begin{align*}
\gamma^{\widetilde{\sigma_f}\widetilde{\sigma_1}} &= \left((-1)^\varepsilon \sqrt{\gamma}\right)^{\widetilde{\sigma_1}} = (-1)^{\varepsilon} \sqrt{f}\sqrt{\gamma}k\\
\gamma^{\widetilde{\sigma_1}\widetilde{\sigma_f}} &= \left(\sqrt{f}\sqrt{\gamma}k\right)^{\widetilde{\sigma_f}} = (-\sqrt{f})((-1)^{\varepsilon} \sqrt{\gamma})(k) =  (-1)^{\varepsilon+1}\sqrt{f} \sqrt{\gamma}k \\
\end{align*} Hence we have $\sqrt{\gamma}^{[\widetilde{\sigma_f},\widetilde{\sigma_1}]} = -\sqrt{\gamma}$ as desired.  Analogous arguments verify $[\widetilde{\sigma_g},\widetilde{\sigma_2}]=\omega$, and that $[\widetilde{\sigma_f},\widetilde{\sigma_2}] = [\widetilde{\sigma_g},\widetilde{\sigma_1}] = \text{\rm id}$.

Our computations have verified that $\Gal(\widetilde{L}/F) \twoheadrightarrow \Gal(L/F)$ is a solution to the embedding problem corresponding to the central, non-split group extension  
$$\widetilde{G} = \left\langle \omega,\widetilde{\sigma_1},\widetilde{\sigma_2},\widetilde{\sigma_f},\widetilde{\sigma_g} \Big{|} \widetilde{\sigma_1}^2 = \omega^{c_1}, \widetilde{\sigma_2}^2 = \omega^{c_3}, [\widetilde{\sigma_f},\widetilde{\sigma_1}]=  [\widetilde{\sigma_g},\widetilde{\sigma_2}]=\omega, [\widetilde{\sigma_1},\widetilde{\sigma_2}] = \omega^{c_2}\right\rangle \twoheadrightarrow \left(\mathbb{Z}/2\mathbb{Z}\right)^{\oplus 4}$$ (where omitted relations are trivial).  By Proposition \ref{prop:Frolich} we therefore have $$(a_1,f)+(a_2,g)+c_1(a_1,a_1)+c_2(a_1,a_2)+c_3(a_2,a_2) = 0$$ in $\Br(F)$.  Since $c_1(a_1,a_1)+c_2(a_1,a_2)+c_3(a_2,a_2) \in \mathcal{S}$, it then follows that $(a_1,f)+(a_2,g) \in \mathcal{S}$ as well.
\end{proof}

\section{The $X$ summand and $\mathcal{S}$}\label{sec:exceptional.summand}

The exceptional summand from Proposition \ref{prop:CMSS.main} is constructed in \cite{CMSS} by examining a particular homomorphism $T:J(K)^G \to \bigoplus_{i=1}^3 (K_i^\times \cap K^{\times 2})/K_i^{\times 2}$ defined by $$T([\gamma]) = ([N_{K/K_1}(\gamma)]_{K_1},[N_{K/K_2}(\gamma)]_{K_2},[N_{K/K_3}(\gamma)]_{K_3}).$$  Note that since $\gamma \in J(K)^G$ we have $N_{K/K_i}(\gamma) \in K^{\times 2} \cap K_i^\times$.  By Kummer theory, this means that $[N_{K/K_1}(\gamma)]_{K_1} \in \{[1]_{K_1},[a_2]_{K_1}\}$, that $[N_{K/K_2}(\gamma)]_{K_2} \in \{[1]_{K_2},[a_1]_{K_2}\}$, and that $[N_{K/K_3}(\gamma)]_{K_3} \in \{[1]_{K_3},[a_1]_{K_3}\}$.  In order to make outputs of $T$ more legible, from here on out we will suppress the bracket-and-subscript notation for each of its components; instead, we ask the reader to interpret each coordinate as an element in the appropriate $J(K_i)$.

The power of the function $T$ is that it can be used as a way to distinguish those elements of $J(K)^G$ which come from the ``obvious'' fixed submodule $[F^\times]$.  To prove this, one considers the Galois group of the extension $K(\sqrt{\gamma})/F$ and computes how lifts of $\sigma_1$ and $\sigma_2$ behave in this group.  Specifically, we have

\begin{lemma}[{\cite[Lem.~4.1]{CMSS}}]\label{le:group.elements.and.T}
Suppose that $[\gamma] \in J(K)^G$ is nontrivial, and let $L = K(\sqrt{\gamma})/F$.  Write $\widetilde{\sigma_1},\widetilde{\sigma_2}$ for lifts of $\sigma_1,\sigma_2$ to $\Gal(L/F)$.  Then
\begin{align*}
[N_{K/K_1}(\gamma)]_{K_1}=[1]_{K_1} &\Leftrightarrow \widetilde{\sigma_2}^2 = \text{\rm id}\\
[N_{K/K_2}(\gamma)]_{K_2}=[1]_{K_2} &\Leftrightarrow \widetilde{\sigma_1}^2 = \text{\rm id}\\
[N_{K/K_3}(\gamma)]_{K_3}=[1]_{K_3} &\Leftrightarrow \left(\widetilde{\sigma_1} \widetilde{\sigma_2}\right)^2 = \text{\rm id}.
\end{align*}
\end{lemma}

The exceptional summand $X$ is ultimately constructed as a module so that $T(X^G) = \text{\rm im}(T)$, but in such a way that $X^G \cap [F^\times]$ is ``as small as possible.''  The details are not particularly important for our discussion.  Instead, the critical fact comes from \cite[Th.~3]{CMSS}, which says that $X$ is nontrivial precisely when $\dim(\text{\rm im}(T))\geq 1$, and that
$$  
X \simeq \left\{\begin{array}{ll}
\mathbb{F}_2,&\text{ if $\dim(\im(T)) = 1$}\\
\Omega^{-1},&\text{ if $\dim(\im(T)) = 2$ and $\im(T)$ is one of the ``coordinate planes''}\\
\mathbb{F}_2\oplus \mathbb{F}_2,&\text{ if $\dim(\im(T)) = 2$ and $\im(T)$ is not one of the ``coordinate planes''}\\
\Omega^{-2},&\text{ if $\dim(\im(T)) = 3$ and $T\left([N_{K/K_1}(K^\times)] \cap [N_{K/K_2}(K^\times)] \right) \neq \{(1,1,1)\}$}\\
\Omega^{-1}\oplus\Omega^{-1},&\text{ if $\dim(\im(T)) = 3$ and $T\left([N_{K/K_1}(K^\times)] \cap [N_{K/K_2}(K^\times)]\right) = \{(1,1,1)\}$.}
\end{array}\right.
$$
The proof of Theorem \ref{th:main.exceptional}, therefore, requires one to draw a connection between $\mathcal{S}$ and $\text{\rm im}(T)$.  The result which delivers this connection is the following

\begin{proposition}\label{prop:T.and.V}
Let $T$ be defined as above.  Then 
\begin{align*}
(a_2,1,1) \in \text{\rm im}(T) &\Leftrightarrow (a_1,a_2)+(a_2,a_2)=0;&&
(1,1,a_1) \in \text{\rm im}(T) &\Leftrightarrow (a_1,a_2)=0;\\
(1,a_1,1) \in \text{\rm im}(T) &\Leftrightarrow (a_1,a_1)+(a_1,a_2)=0;&&
(a_2,1,a_1) \in \text{\rm im}(T) &\Leftrightarrow (a_2,a_2)=0;\\
(a_2,a_1,1) \in \text{\rm im}(T) &\Leftrightarrow (a_1,a_1)+(a_2,a_2)=0;&&
(1,a_1,a_1) \in \text{\rm im}(T) &\Leftrightarrow (a_1,a_1)=0;\\
(a_2,a_1,a_1) \in \text{\rm im}(T) &\Leftrightarrow (a_1,a_1)+(a_1,a_2)+(a_2,a_2)=0.
\end{align*}
\end{proposition}

\begin{proof}
Each of these statements is proved in the same way.  We will focus on the proof of the first result.

First assume that there exists some $[\gamma] \in J(K)^G$ so that $T([\gamma]) = (a_2,1,1)$.  Note that since $[\gamma] \in J(K)^G$ we have that $K(\sqrt{\gamma})/F$ is Galois.  Let $\widetilde{\sigma_1}$ and $\widetilde{\sigma_2}$ be lifts of $\sigma_1,\sigma_2 \in \Gal(K/F)$, and let $\Gal(K(\sqrt{\gamma})/K) = \langle \omega\rangle$.  We have that $\omega \in Z(\Gal(K(\sqrt{\gamma})/F))$.  Furthermore, by Lemma \ref{le:group.elements.and.T} we also have $$\widetilde{\sigma_2}^2 = \omega, \widetilde{\sigma_1}^2 = \text{\rm id}, \left(\widetilde{\sigma_1}\widetilde{\sigma_2}\right)^2=\text{\rm id}.$$  Notice that these relations imply $$[\widetilde{\sigma_1},\widetilde{\sigma_2}] = \widetilde{\sigma_1}^{-1}\widetilde{\sigma_2}^{-1}\widetilde{\sigma_1}\widetilde{\sigma_2} = \widetilde{\sigma_1}\widetilde{\sigma_2}^3\widetilde{\sigma_1}\widetilde{\sigma_2} = \widetilde{\sigma_1}\widetilde{\sigma_2}\omega\widetilde{\sigma_1}\widetilde{\sigma_2} = \omega\left(\widetilde{\sigma_1}\widetilde{\sigma_2}\right)^2 = \omega.$$ Since $K(\sqrt{\gamma})/F$ solves the embedding problem corresponding to the central, non-split group extension $$\left\langle \widetilde{\sigma_1},\widetilde{\sigma_2},\omega \Big{|} \widetilde{\sigma_1}^2 = \text{\rm id}, \widetilde{\sigma_2}^2 =  [\widetilde{\sigma_1},\widetilde{\sigma_2}]=\omega \right\rangle \twoheadrightarrow \langle \sigma_1,\sigma_2\rangle
,$$ Proposition \ref{prop:Frolich} tells us that $(a_1,a_2)+(a_2,a_2)=0.$

Now suppose that $(a_1,a_2)+(a_2,a_2)=0$.  By Proposition \ref{prop:Frolich}, the embedding problem corresponding to the central, non-split group extension $$\left\langle \widetilde{\sigma_1},\widetilde{\sigma_2},\omega \Big{|} \widetilde{\sigma_1}^2 = \text{\rm id}, \widetilde{\sigma_2}^2 =  [\widetilde{\sigma_1},\widetilde{\sigma_2}]=\omega \right\rangle \twoheadrightarrow \langle \sigma_1,\sigma_2\rangle$$ is solvable over $\Gal(K/F)$.  Let $L/F$ be an extension which solves this embedding problem; since $L$ is a quadratic extension of $K$, we know that there exists some $\gamma \in K$ with $L = K(\sqrt{\gamma})$.  Since $L/F$ is Galois, this implies that $[\gamma] \in J(K)^G$. Notice that the relations amongst the generators of $\Gal(L/F)$ imply 
$$\left(\widetilde{\sigma_1}\widetilde{\sigma_2}\right)^2 = \widetilde{\sigma_1}\widetilde{\sigma_2}\widetilde{\sigma_2}\widetilde{\sigma_1}\omega=\widetilde{\sigma_1}\omega\widetilde{\sigma_1}\omega = \widetilde{\sigma_1}^2\omega^2 = \text{\rm id}.$$   An appeal to Lemma \ref{le:group.elements.and.T} then gives $T([\gamma]) = (a_2,1,1)$, as desired.
\end{proof}

\section{Some realizability results}\label{sec:realizability}

We now turn our attention to the question of whether the given summand types we see in Proposition \ref{prop:CMSS.main} actually occur.  Some partial results in this vein have already been discussed in \cite{CMSS}, where it is shown that many of the possible structures of the exceptional summand are realizable.  Ultimately we will construct biquadratic extensions of $\mathbb{Q}$ which exhibit each of the unexceptional summand types: we find nontrivial $Y_0$ and $Z_1$ summands in Example \ref{ex:first}, a nontrivial $Z_2$ summand in Example \ref{ex:second}, and nontrivial $Y_1,Y_2,Y_3,$ and $Y_4$ summands in Example \ref{ex:third}.

Since our extensions will be biquadratic extensions of $\mathbb{Q}$, we will be interested in elements of $\text{Br}(\mathbb{Q})$. To analyze such elements, we will make repeated use of the local-global principle: an element $(a,b) \in \text{Br}(\mathbb{Q})$ is trivial if and only if the local Hilbert symbol $(a,b)_v \in \{\pm 1\}$ is trivial for each place $v \in \{\infty,2,3,5,\cdots\}$ (see \cite[Sec.~18.4]{Pierce}).  Computation of each $(a,b)_v$ are then carried out using well-established formulae.  We will stick with additive notation for elements of $\text{Br}(\mathbb{Q})$, but use multiplicative notation for local Hilbert symbols. In particular, the trivial element in $\text{\rm Br}(\mathbb{Q})$ is denoted $0$, but an element $(a,b) \in \mathbb{Q}$ is trivial at a local place $v$ if $(a,b)_v=1$.

\begin{example}\label{ex:first}
Consider $K/F$ given by $\mathbb{Q}(\sqrt{7},\sqrt{-5})/\mathbb{Q}$.  First we show that $\dim \mathcal{S} = 3$.  To verify this, suppose we had some combination $e_1(7,7)+e_2(7,-5)+e_3(-5,-5)$ which vanishes in $\text{Br}(\mathbb{Q})$.  Examining this equality locally at $v = \infty$ and using the fact that $(7,7)_\infty = (7,-5)_\infty = 1$ whereas $(-5,-5)_\infty = -1$, we therefore find $$1 = \left(e_1(7,7)+e_2(7,-5)+e_3(-5,-5)\right)_\infty = (7,7)_\infty^{e_1}(7,-5)_\infty^{e_2}(-5,-5)_\infty^{e_3} = (-1)^{e_3}.$$  Hence we must have $e_3 = 0$.  Employing a similar strategy at the place $v=5$ (where we have $(7,7)_5=(-5,-5)_5=1$, but $(7,-5)_5 = (7,5)_5 = \left(\frac{7}{5}\right) = \left(\frac{2}{5}\right)=-1$) gives $e_2=0$, and at the place $v=7$ (where $(7,7)_7 = \left(\frac{-1}{7}\right)=-1$ but $(7,-5)_7 = \left(\frac{-1}{7}\right)\left(\frac{5}{7}\right)=(-1)^2=1$ and $(-5,-5)_7=1$) gives $e_1=0$.  Hence the only combination $e_1(7,7)+e_2(7,-5)+e_3(-5,-5)$ which can be trivial is the trivial combination, whence $\dim \mathcal{S} = 3$ as desired.

We show that the corresponding module  has non-trivial $Y_0$ and $Z_1$ summands.  The former is not too difficult, since we simply need to find some element from $\mathbb{Q}$ which is not a square in $\mathbb{Q}(\sqrt{7},\sqrt{-5})$, but which is in the image of the norm map.  A quick calculation shows that $$N_{\mathbb{Q}(\sqrt{7},\sqrt{-5})/\mathbb{Q}}(1+\sqrt{7}+\sqrt{-5}) = 141$$ fits the bill, so that $[N_{\mathbb{Q}(\sqrt{7},\sqrt{-5})/\mathbb{Q}}(\mathbb{Q}(\sqrt{7},\sqrt{-5})^\times)]$ is nontrivial.  By Theorem \ref{th:main.unexceptional}, we have $\text{mult}(Y_0) = \dim [N_{\mathbb{Q}(\sqrt{7},\sqrt{-5})/\mathbb{Q}}(\mathbb{Q}(\sqrt{7},\sqrt{-5})^\times)] > 0$.

Next we argue that $\text{mult}(Z_1)>0$ by proving that $A$ is a proper subset of $V$ (with $A$ and $V$ as defined in Equation \ref{eq:Brauer.subspaces}) for our extension.  Note first that for any biquadratic extension $K/F$, if $[f] \in [N_{K/F}(K^\times)]$ then we must have $(a_1,f) = 0 = (a_2,f)$.  (This is why $A$ is always a subspace of $V$.)  Hence to argue that $A$ is a proper subspace of $V$ for the biquadratic extension $\mathbb{Q}(\sqrt{7},\sqrt{-5})/\mathbb{Q}$, we can argue that there exists some (in fact, infinitely many linearly independent) $[f] \in [\mathbb{Q}^\times]$ with the property that $(7,f)$ and $(-5,f)$ are both nontrivial elements of $\langle (7,7),(7,-5),(-5,-5)\rangle$.   

Let $p$ be any prime satisfying $p \equiv 1,4 \pmod{5}$, $p \equiv 1\pmod{4}$, and $p \equiv 1,2,4 \pmod{7}$.  We will show $(-5,-p) = (-5,-5)$ and that $(7,-p) = (7,7)$.  For the former, note that $(-5,-5)_v = -1$ if and only if $v \in \{\infty,2\}$.  On the other hand, $(-5,-p)_v$ is certainly trivial for $v \not\in \{\infty,2,5,p\}$.  But then we have 
$$(-5,-p)_v = \left\{\begin{array}{ll}
(-1,-1)_\infty(-1,p)_\infty(5,-1)_\infty(5,p)_\infty = -1\cdot 1 \cdot 1 \cdot 1=-1,&\text{ if }v = \infty\\
(-1,-1)_2(-1,p)_2(5,-1)_2(5,p)_2 = -1\cdot (-1)^{\frac{p-1}{2}} \cdot 1 \cdot 1=-1,&\text{ if }v = 2\\
(-1,-1)_5(-1,p)_5(5,-1)_5(5,p)_5 = 1\cdot 1 \cdot \left(\frac{-1}{5}\right) \cdot \left(\frac{p}{5}\right)=1,&\text{ if }v = 5\\
(-1,-1)_p(-1,p)_p(5,-1)_p(5,p)_p = 1\cdot \left(\frac{-1}{p}\right) \cdot 1 \cdot \left(\frac{5}{p}\right) = 1,&\text{ if }v = p\\
\end{array}\right.$$
Similar calculations verify that $(7,-p) = (7,7)$.
%
\end{example}

\begin{example}\label{ex:second}
Consider $K/F$ given by $\mathbb{Q}(\sqrt{33},\sqrt{35})/\mathbb{Q}$.  We will argue that for this extension, $\text{mult}(Z_2)>0$.  To do this, we will find elements $[f],[g] \in [\mathbb{Q}^\times]$ so that $(33,f),(35,g) \in \mathcal{S}$, but with $(33,g)=(35,f) \not\in \mathcal{S}$.  We will then have $[f] \in W_B \setminus V$, after which an appeal to Theorem \ref{th:main.unexceptional} gives the desired result for $\rk(Z_2)$.  It will be useful to know that $(33,33)_v = -1$ if and only if $v \in \{3,11\}$, that $(35,35)_v = -1$ if and only if $v \in \{2,7\}$, and that $(33,35)_v = -1$ if and only if $v \in \{3,5,7,11\}$.

Let $p$ be any prime with $p \equiv 2 \pmod{3}$, $p \equiv 3 \pmod{4}$, $p \equiv 2,3 \pmod{5}$, $p \equiv 1,2,4 \pmod{7}$, and $p \equiv 2,6,7,8,10 \pmod{11}$, and let $q$ be any prime with $q \equiv 1 \pmod{3}$, $q \equiv 1 \pmod{4}$, $q \equiv 1,4 \pmod{5}$, $q \equiv 1,2,4 \pmod{7}$, and $q \equiv 1,3,4,5,9 \pmod{11}$.  We will show that $(33,3pq)=(33,33)$, that $(35,7pq)$ is trivial, and that $(33,7pq)=(35,3pq) \not\in \mathcal{S}$.

First we show that $(33,3pq) = (33,33)$, which means we need to show that $(33,3pq)_v =-1$ if and only if $v \in \{3,11\}$.  Certainly we have $(33,3pq)_v = 1$ for any $v \not\in \{2,3,11,p,q\}$, and 
$$(33,3pq)_v = \left\{\begin{array}{ll}
(3,3)_2(3,p)_2(3,q)_2(11,3)_2(11,p)_2(11,q)_2 = (-1)^4 =  1,&\text{ if }v =2\\
(3,3)_3(3,p)_3(3,q)_3(11,3)_3 = \left(\frac{-1}{3}\right)\cdot \left(\frac{p}{3}\right)\cdot \left(\frac{q}{3}\right)\cdot \left(\frac{11}{3}\right) = -1,&\text{ if }v =3\\
(11,3)_{11}(11,p)_{11}(11,q)_{11} = \left(\frac{3}{11}\right)\cdot \left(\frac{p}{11}\right)\cdot \left(\frac{q}{11}\right) =-1,&\text{ if }v =11\\
(3,p)_p(11,p)_p = \left(\frac{3}{p}\right)\left(\frac{11}{p}\right) = -1\cdot -1 = 1,&\text{ if }v =p\\
(3,q)_q(11,q)_q = \left(\frac{3}{q}\right)\left(\frac{11}{q}\right) = (1)^2 = 1,&\text{ if }v =q.\\
\end{array}\right.$$

Now we show that $(35,7pq)$ is trivial, which means we must show that $(35,7pq)_v=1$ for all places $v$.  Certainly this is true for $v \not\in \{2,5,7,p,q\}$, so we compute:
$$(35,7pq)_v = \left\{\begin{array}{ll}
(5,7)_2(5,p)_2(5,q)_2(7,7)_2(7,p)_2(7,q)_2 = (-1)^2 =  1,&\text{ if }v =2\\
(5,7)_5(5,p)_5(5,q)_5 = \left(\frac{7}{5}\right)\cdot \left(\frac{p}{5}\right)\cdot \left(\frac{q}{5}\right) = (-1)^2=1,&\text{ if }v =5\\
(5,7)_7(7,7)_7(7,p)_7(7,q)_7 = \left(\frac{5}{7}\right)\cdot \left(\frac{-1}{7}\right) \cdot \left(\frac{p}{7}\right)\cdot \left(\frac{q}{7}\right) =(-1)^2=1,&\text{ if }v =7\\
(5,p)_p(7,p)_p = \left(\frac{5}{p}\right)\left(\frac{7}{p}\right) = -1\cdot -1 = 1,&\text{ if }v =p\\
(5,q)_q(7,q)_q = \left(\frac{5}{q}\right)\left(\frac{7}{q}\right) = (1)^2 = 1,&\text{ if }v =q.\\
\end{array}\right.$$

Finally, we argue that $(33,7pq)=(35,3pq) \not\in \mathcal{S}$.  Certainly each of these is trivial for places $v \not\in \{2,3,5,7,11,p,q\}$, so we compute:
$$(33,7pq)_v = \left\{\begin{array}{ll}
(3,7)_2(3,p)_2(3,q)_2(11,7)_2(11,p)_2(11,q)_2 = (-1)^4 =  1,&\text{ if }v =2\\
(3,7)_3(3,p)_3(3,q)_3 =\left(\frac{7}{3}\right)\cdot\left(\frac{p}{3}\right)\cdot\left(\frac{q}{3}\right) =  -1,&\text{ if }v =3\\
1,&\text{ if }v =5\\
(3,7)_7(11,7)_7 = \left(\frac{3}{7}\right)\cdot \left(\frac{11}{7}\right)=-1,&\text{ if }v =7\\
(11,7)_{11}(11,p)_{11}(11,q)_{11} = \left(\frac{7}{11}\right)\cdot\left(\frac{p}{11}\right)\cdot\left(\frac{q}{11}\right) = (-1)^2=1,&\text{ if }v =11\\
(3,p)_p(11,p)_p = \left(\frac{3}{p}\right)\cdot\left(\frac{11}{p}\right) =\left(-(-1)\right)^2 = 1,&\text{ if }v =p\\
(3,q)_q(11,q)_q = \left(\frac{3}{q}\right)\cdot\left(\frac{11}{q}\right) = (1)^2 = 1,&\text{ if }v =q\\
\end{array}\right.$$
$$(35,3pq)_v = \left\{\begin{array}{ll}
(5,3)_2(5,p)_2(5,q)_2(7,3)_2(7,p)_2(7,q)_2 = (-1)^2 =  1,&\text{ if }v =2\\
(5,3)_3(7,3)_3 =\left(\frac{5}{3}\right)\cdot\left(\frac{7}{3}\right) =  -1,&\text{ if }v =3\\
(5,3)_5(5,p)_5(5,q)_5=\left(\frac{3}{5}\right)\cdot\left(\frac{p}{5}\right)\cdot\left(\frac{q}{5}\right)=(-1)^2=1,&\text{ if }v =5\\
(7,3)_7(7,p)_7(7,q)_7 = \left(\frac{3}{7}\right)\cdot \left(\frac{p}{7}\right)\cdot\left(\frac{q}{7}\right)=-1,&\text{ if }v =7\\
1,&\text{ if }v =11\\
(5,p)_p(7,p)_p = \left(\frac{5}{p}\right)\cdot\left(\frac{7}{p}\right) =(-1)^2 = 1,&\text{ if }v =p\\
(5,q)_q(7,q)_q = \left(\frac{5}{q}\right)\cdot\left(\frac{7}{q}\right) = (1)^2 = 1,&\text{ if }v =q.\\
\end{array}\right.$$
This verifies that $(33,7pq)=(35,3pq)$.  To see that this element is not in $\mathcal{S}$, notice that each of the generators has the property that the local values at $v=3$ and $v=11$ agree.  Since $(33,7pq)_3 = -(33,7pq)_{11}$, it must be the case that $(33,7pq) \not\in \mathcal{S}$.
\end{example}

\begin{example}\label{ex:third}
Consider $K/F$ given by $\mathbb{Q}(\sqrt{5},\sqrt{41\cdot 61})/\mathbb{Q}$.  It is not hard to check that $(a,b)=0$ for all $a,b \in \{5,41,61\}$; in particular, this gives $\mathcal{S} = \{0\}$.  We will ultimately show that this extension has $\rk(Y_1),\rk(Y_2),\rk(Y_3)$, and $\rk(Y_4)$ all positive.  To do this, it will be useful to note that in this case we have $V=W_B=W_C=W_D$.  To see this is true, we must prove that whenever $[f],[g] \in [\mathbb{Q}^\times]$ are given such that $[f] \in B$ and $[g] \in C$ and $(a_2,f)+(a_1,g) \in \mathcal{S}$, then in fact $[f] \in C$ and $[g] \in B$.  Given $a_1=5$ and $a_2 = 41\cdot 61$ and $\mathcal{S}=\{0\}$, this amounts to showing that whenever $(5,f)=(41\cdot 61,g)=0$ and $(5,g)=(41\cdot 61,f)$, then in fact $(41\cdot 61,f) = (5,g)=0$ as well. 

Suppose that we have $(5,f)=(41\cdot 61,g)=0$, and that $(5,g)=(41\cdot 61,f)$.  Note first that we can assume $f,g \in \mathbb{Z}$ (since we can multiply through by the square of the denominators).  Hence write \begin{align*}
f &= (-1)^{\alpha_1}5^{\beta_1}41^{\gamma_1}61^{\delta_1} p_1\cdots p_s q_1\cdots q_t\\
g &= (-1)^{\alpha_2}5^{\beta_2}41^{\gamma_2}61^{\delta_2}p_1\cdots p_s r_1\cdots r_\ell,
\end{align*} 
where $\alpha_1,\beta_1,\gamma_1,\delta_1,\alpha_2,\beta_2,\gamma_2,\delta_2\in \{0,1\}$ and $5,41,61,p_1,\cdots,p_s,q_1,\cdots,q_t,r_1,\cdots,r_\ell$ are distinct primes.  Examining $(5,f)_v$ at places $v$ corresponding to $p_i$, the equality $(5,f)=0$ implies $(5,f)_{p_i} = \left(\frac{p_i}{5}\right)=1$ for all $1 \leq i \leq s$. 
Now let us analyze $(41\cdot 61,f)=(5,g)$.  We will show that this element is trivial at every place.  Using the fact that $(5,-1)=(5,41)=(5,61)=0$, we know that $(5,g)_v=1$ for all $v \not\in \{5,p_1,\cdots,p_s,r_1,\cdots,r_\ell\}$.  Similarly we have $(41\cdot 61,f)_v=1$ for all $v \not\in \{41,61,p_1,\cdots,p_s,q_1,\cdots,q_t\}$ (this time taking advantage of equalities like $(61,41)=(41,-1)=0$). Hence the only possibility that $(41\cdot 61,f)_v = (5,g)_v \neq 1$ is if $v \in \{p_1,\cdots,p_s\}$.  But observe that $(5,g)_{p_i} = \left(\frac{5}{p_i}\right) = (5,f)_{p_i} = 1$, completing the proof.


Since we have shown that $V = W_B = W_C = W_D$, Theorem \ref{th:main.unexceptional} gives $\rk(Y_1) = \codim(V \hookrightarrow B)$, $\rk(Y_2)=\codim(V\hookrightarrow C)$, and $\rk(Y_3)=\codim(V\hookrightarrow D)$.  Hence we must only show that $V$ is a proper subspace of $B$, $C$, and $D$.  For the first, let $p$ be any prime which is a square modulo $5$ and modulo $61$, but not a square modulo $41$; then we have $(5,p)=0$ but $(41\cdot 61,p) \neq 0$, and so $[p] \in B \setminus V$.  For the second, let $q$ be any prime which is a square modulo $41$ and $61$, but not a square modulo $5$; then we have $(41\cdot 61,q)=0$ but $(5,q) \neq 0$, so that $[q] \in C \setminus V$.  For the third, let $r$ and $s$ be primes that are neither squares modulo $5$ nor modulo $41$, but are squares modulo $61$.  One then verifies that $(5\cdot 41\cdot 61,rs)=0$, but that $(5,rs) \neq 0$ and $(41\cdot 61,rs) \neq 0$.

Finally, we argue that $\rk(Y_4)>0$.  Following our reasoning from two paragraphs above, we see that $[b] \in B$ if and only if when there exists a square-free integer representative $b=(-1)^{\alpha_1} 5^{\beta_1} 41^{\gamma_1} 61^{\delta_1} p_1\cdots p_s$ such that $\left(\frac{p_i}{5}\right)=1$ for all $1 \leq i \leq s$.  Likewise $[c] \in C$ if and only if we can choose a square-free integer representative $c=(-1)^{\alpha_2} 5^{\beta_2} 41^{\gamma_2} 61^{\delta_2} q_1\cdots q_t$ so that $\left(\frac{q_j}{41}\right)\cdot\left(\frac{q_j}{61}\right)=1$ for all $1 \leq j \leq t$, and with the further property that $\prod_{j=1}^t \left(\frac{q_i}{41}\right) = 1 = \prod_{j=1}^t \left(\frac{q_j}{61}\right)$.  Finally, $[d] \in D$ if and only if we can choose a square-free integer representative $d=(-1)^{\alpha_3} 5^{\beta_3} 41^{\gamma_3} 61^{\delta_3} r_1\cdots r_\ell$ with the property that $\left(\frac{r_k}{5}\right)\cdot\left(\frac{r_k}{41}\right)\cdot\left(\frac{r_k}{61}\right)=1$ for all $1 \leq k \leq \ell$, and further with the properties that $\prod_{k=1}^\ell \left(\frac{r_k}{5}\right) = \prod_{k=1}^\ell \left(\frac{r_k}{41}\right) = \prod_{k=1}^\ell \left(\frac{r_k}{61}\right) =1$.  Hence an element $q \in [\mathbb{Q}^\times]$ has the property that $[q] = [b][c][d]$ for $[b] \in B$, $[c] \in C$, and $[d] \in D$ if and only if $q$ is equivalent to an integer whose prime factors (outside of $5$, $41$, and $61$) consist of primes congruent to a square modulo $5$; primes that are simultaneously squares modulo $41$ and $61$; pairs of primes that are simultaneously \emph{not} squares modulo $41$ or $61$; pairs of primes which are not squares modulo $5$ or $41$, but are square modulo $61$; and pairs of primes which are not squares modulo $5$ and $61$, but are squares modulo $41$.  Therefore let $\pi_1\not\in\{5,41,61\}$ be any prime that is neither a square modulo $5$ nor $41$, but is a square modulo $61$, and let $\pi_2 \not\in\{5,41,61\}$ be any prime which is neither a square modulo $5$ nor $61$, but is a square modulo $41$.  By construction we have $[\pi_1],[\pi_2], [\pi_1\pi_2] \not\in B+C+D$, and hence $\codim(B+C+D\hookrightarrow [\mathbb{Q}^\times]) \geq 2$.  By Theorem \ref{th:main.unexceptional}, we therefore have $\rk(Y_4) \geq 1$.  
\end{example}


\end{document}